\documentclass[11pt, reqno]{amsart}
\usepackage{amsmath, amsthm, amscd, amsfonts, amssymb, graphicx, color, mathrsfs}
\usepackage[bookmarksnumbered, colorlinks, plainpages]{hyperref}
\usepackage{mathtools}
\usepackage[all]{xy}
\usepackage[utf8]{inputenc}
\usepackage{babel}
\usepackage{amstext}
\usepackage{amsthm}
\usepackage{xcolor}
\usepackage{amssymb}
\usepackage{enumitem}
\usepackage{setspace}
\usepackage{mathrsfs}
\usepackage{comment}
\usepackage{dsfont}

\usepackage{lipsum}                                       
\usepackage{xargs}  
\usepackage{cite}
\renewcommand\eqref[1]{(\ref{#1})}

\newcounter{example}[section]
\newenvironment{example}[1][]
{\refstepcounter{example}\par\medskip
\noindent\textbf{Example~\theexample. #1} \rmfamily}{\medskip}

\usepackage{mathtools}
\mathtoolsset{showonlyrefs}

\newtheorem{theorem}{Theorem}[section]
\newtheorem{lemma}[theorem]{Lemma}

\newtheorem{proposition}[theorem]{Proposition}
\newtheorem{corollary}[theorem]{Corollary}
\theoremstyle{definition}
\newtheorem{definition}[theorem]{Definition}

\theoremstyle{remark}
\newtheorem{remark}[theorem]{Remark}
\numberwithin{equation}{section}

\def\mg{{\mathbb{G}}}

\newcommand{\vertiii}[1]{{\left\vert\kern-0.25ex\left\vert\kern-0.25ex\left\vert #1 
    \right\vert\kern-0.25ex\right\vert\kern-0.25ex\right\vert}}
\allowdisplaybreaks

\begin{document}
	\setcounter{page}{1}

\title[Poincar\'{e} inequalities on Carnot Groups]{Poincar\'{e} inequalities on Carnot Groups and spectral gap of Schr\"{o}dinger operators}

	\author[M. Chatzakou]{Marianna Chatzakou}
	\address{
		Marianna Chatzakou:
		\endgraf
		Department of Mathematics: Analysis, logic and discrete mathematics
		\endgraf
		Ghent University
		\endgraf
		Krijgslaan 281, building s8, B 9000 Ghent 
        \endgraf
         Belgium
		\endgraf
		{\it E-mail address} {\rm marianna.chatzakou@ugent.be}
		}

	\author[S. Federico]{Serena Federico}
    \thanks{S.Federico is a member of GNAMPA}
	\address{
		Serena Federico:
		\endgraf
		Department of Mathematics
		\endgraf
		University of Bologna
		\endgraf
		piazza di Porta San Donato 5, 40126, Bologna
	\endgraf
		Italy
		\endgraf
		{\it E-mail address} {\rm serena.federico2@unibo.it}}

			\author[B. Zegarlinski]{Boguslaw Zegarlinski}
	\address{
		Boguslaw Zegarlinski:
		\endgraf
		Institut de Mathématiques de Toulouse ; UMR5219
		\endgraf
		Universit{\'e} de Toulouse ; CNRS
		\endgraf
		118 Route de Narbonne, F-31062 Toulouse Cedex 9
        \endgraf
         France
		\endgraf
		{\it E-mail address} {\rm b.zegarlinski@math.univ-toulouse.fr}
		}
		\endgraf

\thanks{We wish to thank Michael Ruzhansky for fruitful discussions and suggestions. It is our pleasure to thank Bernard Helffer whose valuable comments certainly improved the quality of this work.}

   \keywords{Poincar\'e inequalities; Carnot groups; sub-gradient; spectral gap.}
     \subjclass[2020]{35R03; 35A23; 26D10}
	
\begin{abstract} In this work we give a sufficient condition under which the global Poincar\'{e} inequality on Carnot groups holds true for a large family of probability measures absolutely continuous with respect to the Lebesgue measure. Additionally, we show that the global Poincar\'{e} inequality holds true on any Carnot group for a certain choice of a probability measure adapted to the structure of each Carnot group, and whose formula is explicitly given. Consequently, we extend the results of the previous work (M. Chatzakou, S. Federico, and B. Zegarlinski, \textit{$q$-Poincar\'e inequalities on Carnot groups with a filiform Lie algebra}, Potential Anal. 60 (2024)) targeted on filiform Carnot groups to any Carnot group. As a result, the Schr\"{o}dinger operators associated with the density of the considered probability measure have a spectral gap.
	
\end{abstract} \maketitle
	
%\tableofcontents
	
\section{Introduction}\label{Intro}
\nocite{*}
\subsection{H\"ormander sum of squares and Poincar\'e inequalities}
The main question that we address in this work is the determination of the sufficient conditions that the density of a probability measure $\mu$ on a Carnot group $\mathbb{G}$ must satisfy in order to formulate the so-called \textit{``global Poincar\'e''} or $q$-Poincar\'e inequality in this setting, that is an inequality of the form \eqref{def.gl.p.i}. Another aspect of this investigation is whether we can construct a homogeneous quasi-norm giving rise to the global Poincar\'e inequality on any stratified group. The validity of the Poincar\'e inequality will, in turn, allow to decide whether a certain type of Schr\"odinger operator on $\mathbb{G}$ has a spectral gap. 

The Schr\"odinger-type operators we are referring to are defined via a system of vector fields $\mathbb{X}=\{X_1,\cdots,X_{n}\}$ on $\mathbb{R}^n$ that  satisfies H\"ormander's condition, that is operators that take the form 
\begin{equation}\label{Lform}
-\sum_{i=1}^{n}X_{i}^{2}+V(x)\,,\quad x \in \mathbb{R}^n\,,
\end{equation}
where $V$ is a suitable potential. 

In the present paper the operator $\sum_{i=1}^{n}X_{i}^{2}$ is viewed as the sub-Laplacian of a certain Carnot group. However, not all operators of this form – that is those with vector fields satisfying H\"ormander's condition – are necessarily sub-Laplacians of a Carnot group, although they can always be ``approximated'' by sub-Laplacians on a Carnot group of higher topological dimension $N>n$. The latter is, roughly speaking,  the main idea of the celebrated lifting theorem by Rothschild and Stein \cite{RS76}. 
 
 \par We recall that if $\mathbb{G}\equiv \mathbb{R}^n$ is a group with a Lie algebra that can be generated by vector fields $X_j$, $1<j\leq n_1 \leq n$, using a finite number of repeated commutators, then the negative operator of the form
\[
\Delta_{\mathbb{G}}=\sum_{i=1}^{n_1}X_{i}^{2}\,,
\]
is called the \textit{sub-Laplacian} on $\mathbb{G}$, where $n_1$ is the dimension of the first stratum of the Carnot group. In this case, the system $\mathbb{X}$ is the system of generators of the Lie algebra of $\mathbb{G}$. Such operators, thanks to H\"{o}rmander's celebrated result \cite{Hor67}, are the most typical examples of operators that are not elliptic but hypoelliptic. We recall that when the group structure is not available (i.e. when elements of $\mathbb{X}$ are just vector fields on $\mathbb{R}^n$), the system $\mathbb{X}$ that satisfies H\"ormander's condition consists of   (possibly degenerate, non-commutative) vector fields on $\mathbb{R}^n$ such that 
\[
{\rm rank}({\rm Lie}\{\mathbb{X}\}(x))=n\,, \quad x \in \mathbb{R}^n\,,
\]
where ${\rm Lie}\{\mathbb{X}\}$ is the smallest Lie algebra that can be generated by the system $\mathbb{X}$ using repeated commutators.

 The first pioneering result about Poincaré inequalities involving H\"{o}rmander's vector fields was proved by Jerison in \cite{Jer86}. For the reader's convenience, we recall the statement of the aforementioned inequality in Theorem \ref{thm.jerison} below.  
	
\begin{theorem}
	\label{thm.jerison}
Let $\mathbb G$ be any nilpotent Lie group with the Haar measure $dx$. For any $p \in [1,\infty)$ there exists a constant $P_{0}(r)=P_{0}(r,p)$ such that, for every $x\in \mathbb{G}$ and $f \in C^{\infty}(B_r(x))$,
	\begin{equation}\label{Jer.ineq}
	\int_{B_{r}(x)} |f(y)-f_{B_{r}(x)}|^{p}\,dy \leq P_{0}(r) \int_{B_{r}(x)} |\nabla_{\mathbb G}f(y)|^p\,dy\, ,
	\end{equation}
	where $f_{B_{r}(x)}:= \frac{1}{|B_{r}(x)|}\int_{B_{r}(x)}f(y)\,dy$  and $B_{r}(x):=\{y \in \mathbb G : d(x,y) \leq r \}$ is the ball of radius $r$ centered at $x$ with respect to the  Carnot-Carath\'{e}odory distance $d$. 
\end{theorem} 
%\nocite{*}
We shall refer to \eqref{Jer.ineq} as the {\it local} Poincar\'e inequality, since it holds locally on a ball of radius $r$. 

The Poincaré inequalities involving H\"{o}rmander's vector fields that we are interested in in this paper are the {\it global} Poincar\'e inequalities of the form \eqref{Jer.ineq}, where $B_r(x)$ and the Haar measure on $\mathbb{G}$ (which in this case coincides with the Lebesgue measure) are replaced by the whole $\mathbb{G}$ and by a suitable probability measure on $\mathbb{G}$, respectively.   To be precise, for a suitable probability measure such global Poincar\'e inequality takes the form 
\begin{equation}
    \label{def.gl.p.i}
    \mu(|f-\mu(f)|^q) \leq C \mu (|\nabla_{\mathbb{G}}f|^q)\,,\quad q\geq 1,
\end{equation}
for all functions $f$ for which the right hand side is well defined. Here and later on $\mu$ will be a probability measure on a Carnot group $\mathbb{G}$, $\mu(f)$ will denote the integral $f_\mu=\int_{\mathbb{G}}f\,d\mu$, and $\nabla_{\mathbb{G}}$ will be the vector-valued operator $\nabla_{\mathbb{G}}=(X_1,\cdots,X_{n_1})$ which is the sub-gradient on the group $\mathbb{G}$.

The interest for these inequalities is strongly motivated by the fact that when $q=2$ then \eqref{def.gl.p.i} can be written as 
\begin{equation}
    \label{def.spec.gap}
    \mu(|f-\mu(f)|^2)\leq C \mu (f(\mathcal{L} f))\,,
\end{equation}
 for $\mathcal{L}$ as in \eqref{sch.par.case}, which is referred to as the \textit{spectral gap inequality}. The significance of this estimates lies in the fact that, whenever it holds, it implies the existence of a gap at the bottom of the spectrum of the operator $\mathcal{L}$.  Finally, we remark that estimates of the form \eqref{def.spec.gap}  go back to Henri Poincar\'{e} and imply %(or, whenever $\mu(f)=0$  are equivalent to) 
 the exponential convergence of the associated  semigroup $P_t\equiv e^{t\mathcal{L}}$ to the invariant measure $\mu$, see \cite{Roc01}, \cite{CGR10}, \cite{GZ03}.
 
\subsection{Operators with discrete spectra on Carnot groups and the $U$-bounds method}

As previously mentioned there is a  link between functional inequalities and spectral properties of Schr\"odinger operators. In the work \cite{DM05}, using a probabilistic approach, Driver and Melcher show that the heat kernel measure in the Heisenberg group satisfies a spectral gap inequality. Non-probabilistic approaches to the problem were considered in \cite{El09} and \cite{Li06}. 
As for the investigation of the spectra of other classes of operators in relation to functional inequalities, we refer to the recent works of Cipriani \cite{Cip00} and Wang \cite{Wan02}.

\par  Another approach to prove coercive inequalities related to spectral problems was developed by Hebisch and Zegarlinski in \cite{HZ10}. 
Their method, called the \textit{$U$-bounds method}, works on a general metric space $M$ equipped with the (non-)commuting vector fields $\{X_1,\cdots,X_m\}$.  More specifically, the \textit{$U$-bounds method} consists in proving some inequalities, also called \textit{$U$-bounds}, allowing to derive Poincaré and other inequalities with respect to suitable probability measures.
 
 Precisely, in the Carnot group setting, a  \textit{$U$-bound} is  expressed as
 \begin{equation}
     \label{Ubounds.def}
     \int |f|^q g(d)\,d\mu \leq A_q   \int |\nabla_{\mathbb{G}} f|^q\,d\mu+B_q \int |f|^q\,d\mu\,,
 \end{equation}
 where $g$ is a positive unbounded function on $[0,+\infty)$ (that goes to infinity as $d$ goes to infinity),    $d\mu=e^{-U(d)}\,dx$ is a probability measure defined via a suitable unbounded function $U$, and $dx$ denotes the Haar-measure on the group.
 
 In \cite{HZ10} it is shown that, under suitable conditions on $\mu$, one can pass from the \textit{$U$-bounds} inequalities \eqref{Ubounds.def} to the global (that is with respect to a probability measure) $q$-Poincaré inequalities. This method was successfully applied in the case of step 2 groups in \cite{HZ10}, \cite{Ing10}, \cite{BDZ21a},\cite{IKZ11}, \cite{BDZ22} and \cite{CFZ24}, and on Carnot groups of higher step in \cite{BDZ21b}. In particular, this is also the approach we shall use in the current paper.

 \par  The type of relation between spectral analysis and coercive inequalities that we are going to investigate here is as follows:
 \medskip
 
 For $(M,\mu)$ a probability space, and $\mathcal{L}$ a positive self-adjoint operator on $\mathcal{D}(\mathcal{L}) \subset L^2(\mu)$, the operator $\mathcal{L}$ has a spectral gap if and only if there exists a constant $C>0$ so that 
 \begin{equation}
     \label{dir.form}
      \mu(|f-\mu(f)|^2)\leq C \mathcal{E}(f,f)\,,
 \end{equation}
 where   $(\mathcal{E},\mathcal{D}(\mathcal{E}))$ is the Dirichlet form associated to $\mathcal{L}$, that is the closure of the form
 $$\mathcal{E}(f,g)=\mu(f\mathcal{L}g),\quad f,g\in\mathcal{D}(\mathcal{L}). $$ 

Now, for $U$ as in \eqref{Ubounds.def}, that is $U$ being an unbounded function such that $d\mu=e^{-U(d)}dx$ is a probability measure, by taking 
\begin{equation}\label{sch.par.case}
    \mathcal{L}=-\Delta_{\mathbb{G}}+\nabla_{\mathbb{G}}U\cdot \nabla_{\mathbb{G}}\, ,
\end{equation} 
we get
$$\mathcal{E}(f,f)=\mu(f\mathcal{L}f)=\mu (|\nabla_Gf|^2).$$
At the same time, we know that if an operator $\mathcal{L}$  satisfies the relation 
\[
(\mathcal{L} f,f)_{L^2(\mu)}=\int_{\mathbb{G}}|\nabla_{\mathbb{G}}f(x)|^2\,d\mu(x)\,,
\]
then it is an operator of the form \eqref{Lform} which is positive and self-adjoint on $L^2(\mu)$. 
Therefore, by the previous analysis, \eqref{dir.form} becomes
\begin{equation}
    \label{poincGap}
\mu(|f-\mu(f)|^2)\leq C \mu (|\nabla_Gf|^2),\quad f\in\mathcal{D}(\mathcal{L}),
\end{equation}
for $q=2$, where inequality  \eqref{poincGap} is equivalent to the existence of a spectral gap for the operator $\mathcal{L}$ in  \eqref{sch.par.case}.

Let us stress that the $U$-{\it bounds} method was introduced to deal with the validity of coercive inequalities in non doubling measure spaces, which is also the setting considered in this paper. In the non-doubling setting, a global $L^2-L^2$ Poincaré inequality on connected non-compact Lie
groups was proved in \cite{BPV} in the possibly non-unimodular case, and in \cite{RS} in the unimodular case. In both \cite{BPV} and \cite{RS} the result is based on the existence of a Lyapunov function – an approach different from the one we use here.  As for Poincaré inequalities in doubling metric measure spaces, several results have been obtained so far, and we refer the interested reader to \cite{BL18}, \cite{E19} and references therein.

\subsection{The class of probability measures for which the spectral gap inequality holds true} The quadratic form bounds  in \cite{HZ10}, that is the $U$-bounds in \eqref{Ubounds.def} for $q=2$, resemble the ones in the works of Rosen \cite{Ros76} and Adams \cite{Ada79} in the Euclidean setting. In the case of the Heisenberg group (see \cite{HZ10}), the potential $U$ in the probability measure is taken with respect to the Carnot-Carath\'{e}odory distance. In this consideration, several coercive inequalities, including the Log-Sobolev inequality, are proven.  However, in the same work, the authors prove that replacing the Carnot-Carath\'{e}odory distance by any other smooth homogeneous quasi-norm forces the Log-Sobolev inequality to fail. 
\par In the case of the spectral gap inequality that is of interest to us, the potential is given with respect to a homogeneous quasi-norm on the group. The control over the potential that is present in the $U$-bounds is granted by the lower bound of the length of the sub-gradient of the quasi-norm. In general,  it is the control over the potential $U$ that is needed for such types of inequalities; see also [Corollary 4.2.4 \cite{Ing10}] where it was proved that, as in the Euclidean setting,  a Schr\"{o}dinger operator with potential $V$ as in \eqref{sch.par.case} has a discrete spectrum if $V$ grows to infinity in all directions.  We wish to mention that results on the discreteness of the spectrum of operators as in \eqref{sch.par.case} in the case of H-type groups can be found in \cite{I12}, while for similar results when $\mathbb{G}$ is a Métivier group we refer to in \cite{BC17}. Let us say that in the aforementioned papers the discreteness of the spectrum is obtained by considering a potential defined in terms of a norm $N$ on $\mathbb{G}$ being either the Carnot-Charathéodory
norm or the Kaplan norm, while in \cite{BC17} $N$ is the Kaplan norm.
 
\subsection{Main result and spectral gap}
The main results of this work are contained in Theorems \ref{thm.SuffCond2} and \ref{THM:hypothesis}  below. Particularly, in Theorem \ref{thm.SuffCond2} we give sufficient conditions on the probability measure so that global Poincaré inequalities with respect to that measure hold true, while in Theorem \ref{THM:hypothesis} we show that global Poincar\'e inequalities on a Carnot group $\mathbb G$ always hold for a particular choice of a quasi-norm on $\mathbb G$.

\begin{theorem}\label{thm.SuffCond2}
Let $\mathbb{G}$ be a Carnot group of step $r$ on $\mathbb{R}^n$, let $N$ be a homogeneous quasi-norm on $\mathbb{G}$ smooth away from the origin, and let $\mu_p$ be defined as in \eqref{def.measure}. If there exists an index $j_0\in \{1,\ldots, n_1\}$, and a positive integer $\gamma \geq 2$, such that
\begin{equation}\label{grad-condition2}
    |\nabla_\mathbb{G} N(x)| \gtrsim \,\frac{|x_{j_0}|^{\gamma-1}}{N^{\gamma-1}(x)}\,,\quad \forall x\in\mathbb{G},
\end{equation}
  then,  for all $p\geq 2\gamma$, and for $q$ being the conjugate exponent of $p$, there exists $c_0\in(0,\infty)$ such that
\begin{equation}
    \label{qPoinc2}
\mu_p(|f-f_{\mu_p}|^q)\leq c_0\, \mu_p(|\nabla_{\mathbb{G}} f|^q)
\end{equation}
for all functions $f$ which are Lipschitz (with respect to the Carnot-Carath\'eodory distance).
\end{theorem}
\smallskip

A direct consequence of this result is the spectral gap for a suitable operator $\mathcal{L}_p$.
\begin{corollary}\label{cor.spectralgap}
    Under the hypotheses of Theorem \ref{thm.SuffCond2} the positive self-adjoint operator 
    \[
\mathcal{L}_p:=-\Delta_{\mathbb{G}}+apN^{p-1}\nabla_{\mathbb{G}}N\cdot\nabla_{\mathbb{G}}\,,
    \]
    on $L^2(\mu_p)$ has a spectral gap.
\end{corollary}

\begin{theorem}\label{THM:hypothesis}
   Let $\mathbb{G}$ be a Carnot group  on $\mathbb{R}^n$. Then for $N, \gamma$ as in \eqref{smoothN}, $\mu_p$ as in \eqref{def.measure} with $p\geq 2\gamma$, and for $q$ being the conjugate exponent of $p$, there exists $c_0\in(0,\infty)$ such that
\begin{equation}
    \label{qPoinc2}
\mu_p(|f-f_{\mu_p}|^q)\leq c_0\, \mu_p(|\nabla_{\mathbb{G}} f|^q)
\end{equation}
for all functions $f$ which are Lipschitz (with respect to the Carnot-Carath\'eodory distance). Consequently, also the corresponding  operator $\mathcal L_p$, as in Corollary \ref{cor.spectralgap}, has a spectral gap.
\end{theorem}

The general result in Theorem \ref{THM:hypothesis} is a combination of Theorem \ref{thm.SuffCond2} together with a result by Helffer and Nourrigat in \cite{HN85}, where the authors show that there always exists an admissible change of coordinates in $\mathbb G$ so that the vector fields that generate the corresponding Lie algebra admit a particular form. After reducing to this form, one can show that the lower bound \eqref{grad-condition2} is satisfied, allowing global Poincar\'e inequalities to hold as an application of Theorem \ref{thm.SuffCond2}.

\subsection{Organisation of the paper} The paper is organised as follows: in Section \ref{Section1} we provide the reader with the necessary notions around Carnot groups. In Section \ref{section.SufficientCond} we prove Theorem \ref{thm.SuffCond2} about the existence of sufficient conditions for the global Poincar\'{e}/spectral gap inequalities to hold. Those sufficient conditions are related to the group structure and/or to the choice of the homogeneous quasi-norm, that is, in other words, to the choice of the probability measure. Consequently, the spectral analysis for the self-adjoint operators $\mathcal{L}_p$ follows as an immediate application of Corollary \ref{cor.spectralgap}.  The analysis in this section includes that in \cite{CFZ24} as a special case. In particular, from Theorem \ref{thm.SuffCond2} it follows that the global Poincar\'{e} inequality holds true for every member of the family of the Carnot groups studied in \cite{CFZ24} whose Lie algebra is of filiform type.   Finally, in Section \ref{section.Examples} we give other examples of groups where the sufficient conditions described earlier are satisfied. These include many important examples of Carnot groups that are frequently studied in the literature. Most importantly, in Section \ref{section.Examples} we prove our main general result, i.e., the existence of a quasi-norm on any Carnot group $\mathbb G$ so that the global Poincar\'e inequality holds true.

\section{Preliminaries}\label{Section1}
In this section we shall recall some properties of the Lie algebra of a Carnot group $\mathbb{G}$. The homogeneous structure of Carnot groups plays a fundamental role in our analysis, therefore we start by recalling the definition of homogeneous Lie groups, homogeneous functions, and homogeneous differential operators.
%In particular, we will specify the form of the generators of the Lie algebra of $\mathbb{G}$, that is the vector fields belonging to the first stratum. The particular form of the generators will allow us to state some sufficient conditions for a probability measure on a Carnot group to enjoy global $q$-Poincar{\'e} inequalities.  

\begin{definition}[Homogeneous Lie group on $\mathbb{R}^n$]
\label{def.HomG}
Let $\mathbb{G}=(\mathbb{R}^n,\circ)$ be a Lie group on $\mathbb{R}^n$. We say that $\mathbb{G}$ is a {\it homogeneous} Lie group (on $\mathbb{R}^n$) if there exists an $n-tuple$ of positive integers $\sigma=(\sigma_1,\ldots, \sigma_n)$, with $1=\sigma_1\leq\ldots \leq \sigma_n$, such that the dilation
$$\delta_\lambda: \mathbb{R}^n \rightarrow \mathbb{R}^n,\quad
\delta_\lambda(x_1,\ldots,x_n):=(\lambda^{\sigma_1}x_1,\ldots, \lambda^{\sigma_n}x_n)$$
is an automorphism of the group $\mathbb{G}$ for any $\lambda>0$.
We shall denote by $\mathbb{G}=(\mathbb{R}^n,\circ,\delta_\lambda)$ the datum of a homogeneous Lie group on $\mathbb{R}^n$, where $\circ$ is the composition law and $\{\delta_\lambda\}_{\lambda>0}$ is the dilation group.
\end{definition}

\begin{definition}[$\mathbb{G}$-length of a multi-index]
Given a homogeneous Lie group $\mathbb{G}=(\mathbb{R}^n,\circ,\delta_\lambda)$ and a multi-index $\alpha=(\alpha_1,\ldots,\alpha_n)\in(\mathbb{N}\cup\{0\} )^n$, we define the $\delta_\lambda$-length of $\alpha$ as
$$|\alpha|_{\mathbb{G}}:=\langle \alpha,\sigma\rangle=\sum_{i=1}^n \alpha_i \sigma_i\,,$$
where $\sigma_i$ are as in Definition \ref{def.HomG}.
\end{definition}

%By considering $\mathbb{R}^n$ with or without a group law, and taking $\delta_\lambda$ as above, that is a map as in Definition \ref{def.HomG}, the following definitions and propositions hold true. 
%Below we shall denote the $\delta_\lambda$-length of a multi-index $\alpha\in(\mathbb{N}\cup\{0\} )^n$ by $$|\alpha|_\sigma:=\langle \alpha,\sigma\rangle=\sum_{i=1}^n\alpha_i\sigma_i.$$

\begin{definition}[$\delta_\lambda$-homogeneous function of degree $m$]
Let $a$ be a real function on $\mathbb{G}=(\mathbb{R}^n,\circ,\delta_\lambda)$, we say that $a$ is a $\delta_\lambda$-homogeneous function of degree $m\in\mathbb{R}$ if $a\not\equiv 0$ and for any $x\in \mathbb{G}\setminus 0$ and $\lambda>0$, 
$$ a(\delta_\lambda(x))=\lambda^m a(x).$$
\end{definition}

\begin{definition}[$\mathbb{G}$-degree of a polynomial function]
Given a homogeneous Lie group $\mathbb{G}=(\mathbb{R}^n,\circ, \delta_\lambda)$, and a polynomial function $p:\mathbb{G}\rightarrow \mathbb{R}$ defined (as a {\it finite} sum) as follows
$$p(x)=\sum_\alpha c_\alpha x^\alpha,\quad c_\alpha\in\mathbb{R},$$
we shall call the $\mathbb{G}$-degree of $p$ the quantity
$$\mathrm{deg}_{\mathbb{G}}(p):=\mathrm{max}\{|\alpha|_{\mathbb{G}}: c_\alpha\neq 0\}.$$
\end{definition}

\begin{proposition}[Smooth $\delta_\lambda$-homogeneous functions]\label{prop.smooth.hom}
Let $\mathbb{G}=(\mathbb{R}^n,\circ,\delta_\lambda)$ and $a\in C^\infty(\mathbb{G};\mathbb{R})$. Then $a$ is  $\delta_\lambda$-homogeneous of degree $m$ if and only if it is a polynomial function. As a consequence, the set of the degrees of the smooth (non-vanishing) $\delta_\lambda$-homogeneous functions is
$$\mathcal{A}=\{|\alpha|_\mathbb{G}:\alpha\in (\mathbb{N}\cup\{0\} )^n\}, \quad |\alpha|_\mathbb{G}:=\sum_{i=1}^n\sigma_i \alpha_i,$$
with $|\alpha|_\mathbb{G}=0$ if and only if $a$ is constant.
\end{proposition}

\begin{remark}
If a function $a$ is smooth and $\delta_\lambda$-homogeneous of degree $m$, then $m\geq 0$. Additionally, given a multi-index $\alpha$, one has

$$ D^\alpha a(x)=
\left\{
\begin{array}{ll}
     0 & \,, \forall \alpha\,\, \text{such that}\,\, |\alpha|_\mathbb{G}>m, \\
     \delta_\lambda\text{-homogeneous of degree $m-|\alpha|_\mathbb{G}$ }&\,, \forall \alpha \,\,\text{such that}\,\, |\alpha|_\mathbb{G}\leq m.
\end{array}
\right.
$$
\end{remark}

\begin{definition}[$\delta_\lambda$-homogeneous vector field of degree $m$] Let $\mathbb{G}=(\mathbb{R}^n,\circ,\delta_\lambda)$. Then,
given a vector field $X\in\mathfrak{g}=\mathrm{Lie}(\mathbb{G})$, we say that $X$ is $\delta_\lambda$-homogeneous of degree $m$ if, for any $\varphi\in C^\infty(\mathbb{G})$, $x\in \mathbb{G}\setminus \{0\}$ and $\lambda>0$, 
$$ X(\varphi(\delta_\lambda(x))=\lambda^m (X\varphi)(\delta_\lambda(x)).$$
\end{definition}

\begin{proposition}[Smooth $\delta_\lambda$-homogeneous vector fields]\label{prop.form.lvf}
Let $\mathbb{G}=(\mathbb{R}^n,\circ,\delta_\lambda)$ and $X\in\mathfrak{g}=\mathrm{Lie}(\mathbb{G})$, that is
$$X=\sum_{j=1}^n a_j(x)\partial_{x_j}.$$
Then $X$ is $\delta_\lambda$-homogeneous of degree $k\in\mathbb{R}$ if and only if for every $j=1,\ldots,n$, $a_j$ is a polynomial function $\delta_\lambda$-homogeneous of degree $\sigma_j-k$ (unless $a_j\equiv 0$). Hence the degree of $\delta_\lambda$-homogeneity of $X$ belongs to every set
$$\mathcal{A}_j=\{\sigma_j-|\alpha|_\mathbb{G} : \alpha \in (\mathbb{N}\cup\{0\} )^n\},$$
whenever $j$ is such that $a_j$ is not identically $0$. In other words, for any fixed $j=1,\ldots, n$, $k\leq \sigma_j$ and $k= \sigma_j-|\alpha|_\mathbb{G} $, for some $\alpha\in (\mathbb{N}\cup\{0\} )^n$.
\end{proposition}

\begin{comment}

\begin{remark}
Note that,  by taking $\delta_\lambda(x)=(\lambda^{\sigma_1}x_1,\ldots, \lambda^{\sigma_n}x_n)$, the partial derivatives $\partial_{x_j}$ are $\delta_\lambda$-homogeneous vector fields of degree  $\sigma_j$. This, in particular, explains why a vector field as in Proposition \ref{prop.form.lvf} is $\delta_\lambda$-homogeneous of degree $k$ if all the $a_j(x)$ are homogeneous polynomial of degree $\sigma_j-k$. Indeed,  
$$ Xf(\delta_\lambda(x))=\sum_{j=1}^n a_j(x)\partial_{x_j}f(\delta_\lambda(x))=\lambda^{\sigma_j}\sum_{j=1}^n a_j(x)(\partial_{x_j}f)(\delta_\lambda(x))$$
$$=\lambda^{k} \sum_{j=1}^n \lambda^{\sigma_j-k}a_j(x)(\partial_{x_j}f)(\delta_\lambda(x))=\lambda^k (Xf)(\delta_\lambda(x)),$$
which proves part of Proposition \ref{prop.form.lvf}. For the complete proof see Proposition 1.3.5 in \cite{BLU07}.
\end{remark}
\end{comment}
%When $\mathbb{R}^n$ is equipped with a homogeneous Lie group structure, we use the following definition to indicate the homogeneous degree of a polynomial function.

For more details on homogeneous Lie groups and for the proof of the above propositions, we refer the interested reader to \cite{BLU07}, \cite{CG},\cite{FR16} and \cite{FS82}.
\medskip

We shall now proceed with the formal definition of Carnot groups.  We use the following notation: For $V,W$ vector spaces such that $V,W\subseteq \mathfrak{g}$, we denote by \[[V,W]:=\{[v,w]:v \in V, w \in W\},\] 
where $[\cdot, \cdot]$ is the Lie bracket.
 \begin{definition}\label{defn.Car.hom}
	 Let $\mathbb{G}=(\mathbb{R}^n, \circ)$ be a Lie group on $\mathbb{R}^n$, and let $\mathfrak{g}$ be the Lie algebra of $\mathbb{G}$. Then $\mathbb{G}$ is called a stratified group, or \textit{Carnot group}, if $\mathfrak{g}$ admits a vector space decomposition (stratification) of the form
	 \begin{equation}
	 \label{def.carnot}
	 \mathfrak{g}=\bigoplus_{j=1}^{r} V_j\,,\quad \text{such that}
	 \quad 
	 \left\{
	 \begin{array}{l}
	 [V_1,V_{i-1}]= V_{i}\, , 
	 \quad 2\leq i\leq r,\\
	 
	 [V_1,V_r]=\{0\},\\
	 \end{array}
	 \right.
	 \end{equation}
	 where the positive integer $r$ is called the {\it step} of $\mathbb{G}$.
	 \end{definition}

\begin{remark}
Carnot groups are naturally homogeneous Lie groups. Therefore, each Carnot group $\mathbb{G}$ can be equipped with a mapping $\delta_\lambda$, $\lambda>0$, as in Definition \ref{def.HomG} that is an automorphism of the group $\mathbb{G}$. The diffeomorphic map $\exp_{\mathbb{G}}: \mathfrak{g}\rightarrow \mathbb{G} $ ensures that $\delta_\lambda$ is also an automorphism of $\mathfrak{g}$, where $\mathfrak{g}$ is the Lie algebra of $\mathbb{G}$. Note that the  stratification of a Lie algebra $\mathfrak{g}$ is not unique. However, the mapping $\delta_\lambda$ as in Definition \ref{def.HomG} does not depend on the  choice of the stratification; cf. [Proposition 2.2.8 \cite{BLU07}]. We can then denote by $\mathbb{G}=(\mathbb{R}^n,\circ,\delta_\lambda)$ the Carnot group equipped with the natural dilations $\{\delta_\lambda\}$.
%\textcolor{blue}{Do we need this?}\textcolor{purple}{Indeed, given a stratification $\mathfrak{g}=\oplus_{j=1}^{r}V_j$, where each $V_j$ consists of $n_j \neq 0$ elements of $\mathfrak{g}$, and writing $x \in \mathbb{G}$ as
%	 \[
%	 x=(x^{(n_1)},\cdots,x^{(n_{r})})\,, \quad \text{where}\quad x^{(n_j)}\in \mathbb{R}^{n_j}\,,
%	 \]
%	 we have that the mapping $\delta_\lambda:\mathbb{R}^n \rightarrow \mathbb{R}^n$, $\lambda>0$, defined by
	 $$
%	 \delta_{\lambda}(x):=(\lambda x^{(n_1)},\cdots,\lambda^{r}x^{(n_{r)}})$$
%	 $$=(\lambda^{\sigma_1}x_1,\ldots, \lambda^{\sigma_{n_1}}x_{n_1},\ldots, 
%	 \lambda^{\sigma_{n-n_r+1}}x_{n-n_r+1},\ldots,\lambda^{\sigma_n}x_n) 
	 $$
%	 is an automorphism of $\mathbb{G}$ for every $\lambda>0$.}
	 \end{remark}

We shall conclude this section by recalling the definition of  sub-Laplacian and   sub-gradient on a Carnot group $\mathbb{G}$.

 \begin{definition}
	 \label{defn.gen.op.subl-g}
	 Let $\mathbb{G}$ be a Carnot group, and let $\mathfrak{g}$ be the corresponding Lie algebra. If $X_j$, $1 \leq j \leq n_1$, are the canonical left (right) invariant vector fields that generate $\mathfrak{g}$, then the second order differential operator 
	 \[
	 \Delta_{\mathbb{G}}=\sum_{j=1}^{n_1}X_{j}^{2}\,,
	 \]
 is called the \textit{canonical left (right) invariant sub-Laplacian} on $\mathbb{G}$, while the vector valued operator 
 \[
 \nabla_{\mathbb{G}}=(X_1,\cdots,X_{n_1})\,,
 \]
 is called the \textit{canonical left (right) invariant $\mathbb{G}$-gradient}. 
	 \end{definition}
\begin{definition}\label{def.hom.n}
We call \textit{homogeneous (quasi-)norm} on the Carnot group $\mathbb{G}$, every continuous, with respect to the Euclidean topology, mapping $N: \mathbb{G} \rightarrow [0,\infty)$ such that \[ N(x)=0 \text{ if and only if}\quad  x = 0,\] and 
\[
N(\delta_{\lambda}(x))=\lambda N(x)\,,\quad \text{for every}\quad \lambda>0\,,\quad x \in \mathbb{G}\,.
\]
\end{definition}

\section{Sufficient conditions for global Poincar{\'e} inequalities: proof of Theorem \ref{thm.SuffCond2}}
%{on Carnot groups}
\label{section.SufficientCond}

 %For a probability measure $\mu$ on Carnot group $\mathbb{G}$ and an integrable function $f$, we use a notation $f_\mu:=\mu(f)\equiv \int f d\mu$.\\
 
  %A probability measure $\mu$ on the group $\mathbb{G}$ satisfies   a $q$-Poincar{\'e} inequality, $q\geq 1$, if and only if  there exists a   constant $C\in(0,\infty)$ such that

%$$\mu ( |f-f_\mu|^q)\leq C \mu (|\nabla_\mathbb{G}f|^q),$$
 %for all functions $f$ for which the integral on the right hand side is well defined.

In this section, we establish {\it sufficient} (though not necessary) conditions for global $q$-Poincar{\'e} inequalities to hold for some 
probability measures on Carnot groups, by proving Theorem \ref{thm.SuffCond2}.
 The probability measures suitable for our purposes - denoted by $\mu_p$ and defined below - have densities (with respect to the Haar measure) depending on a parameter $p$ and on a fixed homogeneous quasi-norm $N$ on $\mathbb{G}$. As for the necessity of our conditions, it is not known yet if they are also necessary to obtain Poincaré inequalities with respect to probability measures of the form considered here.

%We shall start by giving the definition of the homogeneous quasi-norm $N$ that we will use throughout the paper.

%\begin{definition}[Homogeneous quasi-norm $N$] \label{defn.N}
%Let $G=(\mathbb{R}^n,\circ,\delta_\lambda)$ a Carnot group. Then, the mapping $N:\mathbb{G} \rightarrow [0,\infty)$ is a homogeneous quasi-norm on $\mathbb{G}$ if it is a continuous function on $\mathbb{G}$ satisfying:
%\begin{enumerate}
 %   \item $N(\delta_\lambda(x))=\lambda N(x)$ for every $\lambda>0$ and for every $x \in \mathbb{G}$;
  %  \item  $N(x)=0$ iff $x=0_{\mathbb{G}}$.
%\end{enumerate}
%\end{definition}

\medskip

Given  a homogeneous quasi-norm $N$,   we define the probability measure $\mu_p$ as  %shall call $\mu_p$ the measure
\begin{equation}\label{def.measure}
 \mu_p:= \frac{1}{Z}e^{-a N^p}dx,   
\end{equation}
where $Z$ is a normalization constant, $a\in(0,\infty)$, and $p\in (0,\infty)$.
%being a suitable real number. 
As we will explain later, all our results will hold for a large class of perturbations of such measures.

{Here and later on we use a convention $a  \gtrsim b$ (resp. $a \lesssim b$) to say that there exists a constant $C\in(0,\infty)$ independent of $a,b$ so that $a \geq C b$ (resp. $a \leq C b$).} Also, for $p,q\in[1,\infty]$ we say that   $q$ is the conjugate exponent of $p$ iff $\frac1p+\frac1q=1$.

%\begin{remark}
%\textcolor{purple}{Should we write here that N given as in 3.16 if $X_j$ is an in Proposition 3.16 then (3.9) is satisfies? And then write here what follows after Prop. 3.16?  }
%\end{remark}

%{\color{blue} \\ \vspace{0.25cm} 
%Working Remark  : Change initial $f$ in in definition of $N$ to something else because later
%$f$ is used in Poincare inequality.\\ \vspace{0.25cm} }

\begin{comment}
    
\begin{theorem}\label{thm.SuffCond2}
Let $\mathbb{G}$ be a Carnot group of step $r$ on $\mathbb{R}^n$, let $N$ be a homogeneous quasi-norm on $\mathbb{G}$ smooth away from the origin, and let $\mu_p$ be defined as in \eqref{def.measure}. If there exists an index $j_0\in \{1,\ldots, n_1\}$, and a positive integer $\gamma \geq 2$, such that
\begin{equation}\label{grad-condition2}
    |\nabla_\mathbb{G} N| \gtrsim \,\frac{|x_{j_0}|^{\gamma-1}}{N^{\gamma-1}(x)}\,,
\end{equation}
  then,  for all $p\geq 2\gamma$, and for $q$ being the conjugate exponent of $p$, there exists $c_0\in(0,\infty)$ such that
\begin{equation}
    \label{qPoinc2}
\mu_p(|f-f_{\mu_p}|^q)\leq c_0\, \mu_p(|\nabla_{\mathbb{G}} f|^q)
\end{equation}
for all functions $f$ which are Lipschitz (with respect to the Carnot Carath\'eodory distance).
\end{theorem}
\end{comment}

We remark that the requirement that $f$ is Lipschitz, with respect to the Carnot-Carath\'eodory distance, guarantees that the sub-gradient $\nabla_{\mathbb{G}}f$ exists a.e., see Theorem 2.5 in \cite{MSC01}. Moreover, by Theorem 4.2.4 in \cite{HKST15} such functions are dense in $L^q(\mu)$, for $q \in [1, \infty)$.

%\begin{lemma}\label{lem.suf.cond2}
%Let $\mathbb{G}$ be a Carnot group of step $r$ on $\mathbb{R}^n$ and let $N$ be a homogeneous quasi-norm on $\mathbb{G}$ smooth away from the origin such that \eqref{grad-condition2} holds. Then
%\begin{equation}\label{grad-condition3}
 %    |\nabla_\mathbb{G} N(x)|\geq C\, \frac{|x_{j_0}|^{\gamma-1}}{N^{\gamma-1}(x)}\,, \end{equation}
  %   for some $j_0 \leq n_1$ and $\gamma$ being the same as in \eqref{grad-condition2}.
%\end{lemma}

%\begin{remark}
%In the proof of Theorem \ref{thm.SuffCond2} we will actually use Lemma \ref{lem.suf.cond2}, that is condition \eqref{grad-condition3} instead of condition \eqref{grad-condition2}. Note that, if $k>1 $, \eqref{grad-condition2} implies \eqref{grad-condition3}, while when $k=1$, the two conditions essentially coincide. Therefore, it is not restrictive to ask \eqref{grad-condition2} instead of \eqref{grad-condition3} in Theorem \ref{thm.SuffCond2}.
%\end{remark}

\begin{remark}[Regarding the condition \eqref{grad-condition2}] 
Observe that condition \eqref{grad-condition2} implies that 
$$\mathrm{Ker}( |\nabla_\mathbb{G}N|):= \{x\in\mathbb{R}^n: |\nabla_{\mathbb{G}} N(x)|=0 \}\subseteq \{x\in\mathbb{R}^n: x_{j_0}=0 \}.$$
In other words, a necessary (but not sufficient) condition to have \eqref{grad-condition2} is that $\mathrm{Ker}( |\nabla_\mathbb{G}N|)$ is contained in the hyperplane $\{x\in\mathbb{R}^n: x_{j_0}=0 \}$.

\end{remark}
For the proof of Theorem \ref{thm.SuffCond2} the following two auxiliary lemmas are necessary.

\begin{lemma}\label{lemma0}
Let $\mathbb{G}$ be a Carnot group on $\mathbb{R}^n$, and let $N$ be a homogeneous quasi-norm on $\mathbb{G}$ smooth away from the origin. Then, for all $x\in\mathbb{G}\setminus\{0\}$, we have

\begin{equation}\label{eq.gradest2}
    |\nabla_\mathbb{G} N(x)|\lesssim 1
\end{equation}

\begin{equation}\label{eq.subLdest2}
    |\Delta_\mathbb{G} N(x)|\lesssim \frac{1}{N(x)}.
\end{equation}
\end{lemma}

\begin{proof}
The proof  follows from the homogeneity properties of $N$ and the fact that $|\nabla_{\mathbb{G}}N|^2$ and $\Delta_{\mathbb{G}}N$ are homogeneous functions of degree $0$ and $-1$, respectively.
\end{proof}

\begin{lemma}\label{lemma1.2}
Let $p\geq 2\gamma$ and let $q$ be its conjugate exponent.  Let $N$ be a homogeneous quasi-norm smooth away from the origin satisfying \eqref{grad-condition2} for some $j_0\in \{1,\ldots, n_1\}$, and let $\mu_p:=e^{-aN^p}dx$. Then there exist $A\in(0,\infty)$ and $B\in[0,\infty)$ such that  
\begin{equation}
    \label{est.lemma1.2}
\mu_p(|f|^q N^{p-2\gamma}| x_{j_0}| ^{2\gamma})\leq A\mu_p(|\nabla_\mathbb{G}f|^q)+B\mu_p(|f|^q),
\end{equation}
for every function $f$ which is Lipschitz (with respect to the Carnot-Carath\'eodory distance).
\end{lemma}

\proof

We shall prove the result for $f\geq 0$ such that $f\in C_0^\infty(\mathbb{G})$, which, by approximation, gives the result when $f$ is non-negative and Lipschitz (this can be seen by using the fact that the density of the probability measure has strong decay properties, so $C_0^\infty(\mathbb{G})$ is dense in $L^p(\mu)$ and $\mathrm{Lip}(\mathbb{G})\subset L^p(\mu)$). The result for a general $f$ will then follow from this case by %standard %approximation arguments 
the fact that a.e. one has $|\nabla_\mathbb{G}|f||\leq |\nabla_\mathbb{G}f|$.

By the Leibniz rule we have that
$$e^{-a N^p}\nabla_\mathbb{G}f=
\nabla_\mathbb{G} (e^{-a N^p}f)+ap  N^{p-1} e^{-a N^p}f (\nabla_\mathbb{G}N),$$
therefore, taking the inner product of the vectors $e^{-a N^p}\nabla_\mathbb{G}f$ and $\nabla_\mathbb{G}N$ we have
$$\int \nabla_\mathbb{G}N(x)\cdot (\nabla_\mathbb{G} f(x)) e^{-a N(x)^p} dx$$
\begin{equation}\label{pflemma.1.2}
=\int \nabla_\mathbb{G}N(x)\cdot \nabla_\mathbb{G} (e^{-a N(x)^p}f(x))dx+ ap \int e^{-a N(x)^p}f(x) N^{p-1}(x) |\nabla_\mathbb{G}N(x)|^2dx.
\end{equation}
We then apply Cauchy-Schwarz inequality and \eqref{eq.gradest2} to the left-hand side of \eqref{pflemma.1.2},
integrate by parts  the first term on the right-hand side of \eqref{pflemma.1.2}, and  apply \eqref{grad-condition2} on the second term of the right hand side of it. This gives, for $C_2\geq 1$ and $C_1>0$ such that $|\nabla _\mathbb{G}N(x)|\leq C_1$ (see \eqref{eq.gradest2}), that
\begin{align}
 C_1\int |\nabla_\mathbb{G} f(x)| e^{-a N(x)^p} dx &\geq \int |\nabla_\mathbb{G}N(x)||\nabla_\mathbb{G} f(x)| e^{-a N(x)^p} dx\\
&\geq -\int (\Delta_{\mathbb{G}}N(x)) f(x) e^{-a N(x)^p}dx\\
&+ ap \int e^{-a N(x)^p}f(x) N^{p-1}(x) |\nabla_\mathbb{G}N(x)|^2dx\\
&\geq - C_2 \int |\Delta_{\mathbb{G}}N(x)| f(x) e^{-a N(x)^p}dx\\
&+ ap \,C_0 \int N(x)^{p-2\gamma+1}|x_{j_0}|^{2\gamma-2}f(x) e^{-a N(x)^p}dx.
\label{est.may.22}
\end{align}
Then, moving the term containing $\Delta_\mathbb{G}N$ on the right hand side of \eqref{est.may.22} to the left hand side of \eqref{est.may.22},  we get  
$$ap \,C_0 \int N(x)^{p-2\gamma+1}|x_{j_0}|^{2\gamma-2}f(x) e^{-a N(x)^p}dx$$
$$\leq C_1\int |\nabla_\mathbb{G}f(x)|e^{-a N(x)^p}dx +C_2\int |\Delta_{\mathbb{G}}N(x)| f(x) e^{-a N(x)^p}dx,$$
where $j_0$ is the index appearing in \eqref{grad-condition2} and $C_0,C_1,C_2\in(0,\infty)$ are some constants independent of the function $f$.\\
Since, by homogeneity,
we have $|\Delta_{\mathbb{G}}N(x)|\lesssim 1/N(x)$ and $|x_{j_0}|^{2\gamma-2} \gtrsim |x_{j_0}|^{2\gamma-1}/N(x)$, we get that
$$ap \,C_0 \int N(x)^{p-2\gamma}|x_{j_0}|^{2\gamma-1}f(x) e^{-a N(x)^p}dx$$
$$\leq C_1\int |\nabla_\mathbb{G}f(x)|e^{-a N(x)^p} dx +C_2\int \frac{1}{N(x)} f(x) e^{-a N(x)^p}dx,$$
possibly with new constants that we keep denoting here and later on by $C_0,C_1,C_2$.
 Note that the previous inequality can be rewritten in short notation as 
$$ap\, C_0 \mu_p(f N^{p-2\gamma}|x_{j_0}|^{2\gamma-1})\leq C_1\mu_p(|\nabla_\mathbb{G}f|)+ C_2\mu_p\left( \frac{f}{N}\right).$$
Now (using suitable approximation), we replace $f$ with $f \cdot |x_{j_0}|$ in the inequality above and get
$$ap \,C_0 \int N(x)^{p-2\gamma}|x_{j_0}|^{2\gamma}f(x) e^{-a N(x)^p}dx$$
$$\leq C_1\int |\nabla_\mathbb{G}(f(x)|x_{j_0}|)|e^{-a N(x)^p}dx +C_2\int \frac{|x_{j_0}|}{N(x)} f(x) e^{-a N(x)^p}dx.$$

Using that $|\nabla_\mathbb{G}(f(x)|x_{j_0}|)|\leq|\nabla_\mathbb{G}f(x)|\,|x_{j_0}|+ f(x)$ (since $\left|\nabla |x_{j_0}|\right|=1$ and $f\geq 0$), and that $\frac{|x_{j_0}|}{N(x)}\lesssim 1$, it follows that 
$$ap \,C_0 \int N(x)^{p-2\gamma}|x_{j_0}|^{2\gamma}f(x) e^{-a N(x)^p}dx$$
$$\leq C_1\int |\nabla_\mathbb{G}f(x)|\, |x_{j_0}|e^{-a N(x)^p}dx+C_2\int f(x) e^{-a N(x)^p}dx.$$
Replacing $f$ with $f^q$ in the previous estimate, we get

$$ap \,C_0 \int N(x)^{p-2\gamma}|x_{j_0}|^{2\gamma}f(x)^q e^{-a N(x)^p}dx$$
$$\leq C_1\int qf^{q-1}|\nabla_\mathbb{G}f(x)|\, |x_{j_0}|e^{-a N(x)^p}dx+C_2\int f(x)^q e^{-a N(x)^p}dx.$$
An application of  Young's inequality, for any $\varepsilon>0$, yields
$$ qf^{q-1}|\nabla_\mathbb{G}f|\, |x_{j_0}|\leq \frac{1}{\varepsilon^{q-1}}|\nabla_\mathbb{G}f|^q+\frac{q}{p}\varepsilon |x_{j_0}|^p f^q.$$
Therefore, since $|x_{j_0}|^p\leq N^{p-2\gamma}(x)\cdot |x_{j_0}|^{2\gamma}$, we have 
$$\left(ap C_0-C_1\frac{q}{p}\varepsilon\right)
\mu_p(f^q N^{p-2\gamma}|x_{j_0}|^{2\gamma})
\leq \frac{C_1}{\varepsilon^{q-1}}
\mu_p(|\nabla_\mathbb{G}f|^q)+ C_2\mu_p(f^q).$$
Finally, choosing $\varepsilon>0$ sufficiently small in such a way that the constant on the left hand side is positive, we conclude \eqref{est.lemma1.2} for non-negative compactly supported functions, which implies the result for suitable general $f$.
\endproof

\begin{remark}
Note that in Theorem \ref{thm.SuffCond2} and in the lemmas above we require the quasi-norm $N$ to be smooth away from the origin. This assumption is used to avoid technical problems when applying the vector fields to $N$. However, in the case when $N$ is not smooth on some hyperplane, one can split the domain into connected components where $N$ is differentiable. In such cases, working locally inside each connected component, and using a suitable approximation argument like in \cite{CFZ24}, one can extend the result to the whole domain.
\end{remark}

With the previous lemma at our disposal we can now prove Theorem \ref{thm.SuffCond2}.

\proof[Proof of Theorem \ref{thm.SuffCond2}]
It is easy to check that for all $m\in \mathbb{R}$, a.e. we have $$\mu_p(|f-f_{\mu_p}|^q)\leq 2^q \mu_p(|f-m|^q)\,.$$ So  it is enough to prove the estimate for $\mu_p(|f-m|^q)$ with a suitable choice of $m$.
Then, with some $R>0$ and $L>0$ to be chosen later, we have

$$\mu_p(|f-m|^q)=\mu_p(|f-m|^q \boldsymbol{\mathds{1}} _{\{|x_{j_0}|^{2\gamma}N^{p-2\gamma}\geq R\}}) + \mu_p(|f-m|^q \boldsymbol{\mathds{1}} _{\{|x_{j_0}|^{2\gamma}N^{p-2\gamma}\leq R\}}\boldsymbol{\mathds{1}}_{\{N\leq L\}})$$
\begin{equation}
    \label{prf.thm.1}
+ \mu_p(|f-m|^q\boldsymbol{\mathds{1}}_{\{|x_{j_0}|^{2\gamma}N^{p-2 \gamma}\leq R\}}\boldsymbol{\mathds{1}}_{\{N\geq L\}})
\end{equation}
$$=I+II+III.$$

For the term $I$, by Lemma \ref{lemma1.2}, we have
$$I %=\mu_p(|f-m|^q)
=\mu_p(|f-m|^q \boldsymbol{\mathds{1}}_{\{|x_{j_0}|^{2\gamma}N^{p-2\gamma}\geq R\}})\leq \frac{1}{R}\mu_p(|f-m|^qN^{p-2\gamma} |x_{j_0}|^{2\gamma})$$
$$\leq \frac{A}{R}\mu_p(|\nabla_\mathbb{G}f|^q)+\frac{B}{R}\mu_p(|f-m|^q).$$

To estimate $II$ we first observe that, given any homogeneous quasi-norm $N$, there exists $C\in(0,\infty)$ such that the Carnot-Carath\'{e}odory distance $d(x)=d(x,0)$ satisfies
$$C^{-1} N(x)\leq d(x)\leq C N(x), \quad \forall x\in \mathbb{G}.$$
 Therefore, given $L>0$, there exists $L_1,L_2>0$ such that for the Carnot-Carath\'{e}odory ball $B_{L_1}$ of radius $L_1$ centered at the origin, we have  
\begin{equation}\label{equiv.CC} \{N\leq L\}\subset B_{L_1}\subset \{N\leq L_2\}.\end{equation}
Choosing $m=\frac{1}{|B_{L_1}|}\int_{B_{L_1}} f(x)dx$, we get
\begin{align*}
 II&=\mu_p(|f-m|^q \boldsymbol{\mathds{1}}_{\{|x_{j_0}|^{2\gamma}N^{p-2\gamma}\leq R\}}\boldsymbol{\mathds{1}}_{\{N\leq L\}})\\
 &\leq \mu_p(|f-m|^q\boldsymbol{\mathds{1}}_{\{N\leq L\}})=\frac{1}{Z}\int_{\{N\leq L\}}|f(x)-m|^qe^{-aN(x)^p}dx\\
 &\leq\frac{1}{Z}\int_{\{N\leq L\}}|f(x)-m|^qdx\\
 & \leq\frac{1}{Z}\int_{B_{L_1}}|f(x)-m|^qdx\\
 & \leq\frac{P_0(L_1)}{Z}\int_{B_{L_1}}|\nabla_\mathbb{G}f(x)|^qdx\\
 &\leq\frac{P_0(L_1)e^{aL_2^p}}{Z}\int_{\{N\leq L_2\}}|\nabla_\mathbb{G}f(x)|^qe^{-aN(x)^p} dx\\
 & \leq P_0(L_1)e^{aL_2^p} \mu_p( |\nabla_\mathbb{G}f(x)|^q),
\end{align*}
where in the fourth line we applied the Poincar\'e inequality on balls (see \cite{Jer86}).

We are now left with the estimate of term $III$. To this end define a set
$$A_{L,R}:=\{ x\in \mathbb{G}: |x_{j_0}|^{2\gamma}\leq R, N(x)\geq L\}.$$
We claim that for all $x\in A_{L,R}$  there exist  a positive constant $c'=c'(R)<1$ sufficiently small and a horizontal curve $\gamma_x: [0,t]\mapsto G$ such that 
$\gamma_x(0)=e$, $(x\circ \gamma_x(t))_{j_0}>c' R^{\frac{1}{2\gamma}}$, and
\begin{equation}\label{ineqN}
    R\leq N(x\circ \gamma_x(s))< N(x), \quad \forall s \in (0,t].
\end{equation}
Let us clarify that the reason for the subscript $x$ in $\gamma_x$ is to stress that the choice of the curve with the above properties with respect to $x$ depends on $x$, so $x$ is not the starting point of $\gamma_x$.

 Next, we give the proof of our claim. In the following we shall say that a geodesic $\gamma:[0,t]\rightarrow \mathbb{G}$ connects $x$ and $y$ if $\gamma(0)=e$, hence $x=x\circ \gamma(0)$, and $y=x\circ \gamma(t)$. 
Since $N$ is smooth and there exists $r=r(x)>0$ such that $x\in \partial B_r$, we have $N(x)> N(y)$ for all $y\in B_r$.
Then, we choose $y\in B_r$ and take $\gamma_{1,x}:[0,t]\mapsto \mathbb{G}$  the horizontal geodesic connecting $x$ and $y$, that is such that $x\circ \gamma_{1,x}(0)=x $ and $x\circ \gamma_{1,x}(t)=y$. Note that $N(x\circ \gamma_{1,x}(s))<N(x)$ for all $s\in (0,t]$.
Due to our choice of $y$, which satisfies $N(y)<N(x)$, there exists $c=c(r,R)$, with $|c|<1$ sufficiently small, such that $N(y\circ (0,\ldots, c R^{\frac{1}{2\gamma}},\ldots, 0))=N(y_1,\ldots, y_{j_0}+cR^{\frac{1}{2\gamma}}, y_{j_0+1},\ldots, y_n )<N(x)$
and 
$$y_{j_0}+cR^{\frac{1}{2\gamma}}> c' R^{\frac{1}{2\gamma}}$$
for some $0<c'<1$ small.
Now we call $\gamma_{2,x}$ the horizontal geodesic % geodesic ???
connecting $y$ and $y\circ (0,\ldots, c R^{\frac{1}{2\gamma}},\ldots, 0)$, that is such that $\gamma_{2,x}(0)=e$, $y\circ\gamma_{2,x}(0)=y$ and $y\circ\gamma_{2,x}(t)=y\circ (0,\ldots, c R^{\frac{1}{2\gamma}},\ldots, 0)$, and $\gamma_x:[0,t]\mapsto \mathbb{G}$ the union of $\gamma_{1,x}$ and $\gamma_{2,x}$, more precisely $\gamma_x(s):=\gamma_{1,x}(s)\circ \gamma_{2,x}(s)$, $s\in [0,t]$. Since $\gamma_x$ connects $x$ and $x\circ h(x)=y(x)\circ (0,\ldots, c R^{\frac{1}{2\gamma}},\ldots, 0)$, where $h(x)$ is a suitable point which depends on $x$ and satisfiying the previous identity, we conclude that $\gamma_x$ satisfies \eqref{ineqN} and hence the claim.

Note that, by choosing $L$ much larger than $R$, then for all $x\in A_{L,R}$ we can take $r(x)\sim R$  in the argument above so that $d(x,x\circ h)=d(h)\lesssim R$.
 Hence, using the properties of the curve $\gamma_x$ and the notation $g=f-m$, we get
$$III=\mu_p(|g|^q \boldsymbol{\mathds{1}} _{\{|x_{j_0}|^{2\gamma}N^{p-2\gamma}\leq R\}}\boldsymbol{\mathds{1}}_{\{N\geq L\}})\lesssim\int_{A_{L,R}}|g(x)|^q d\mu_p(x)$$
\begin{equation}\label{NewEst1}
\lesssim \int_{A_{L,R}}|g(x)-g(x\circ h)|^q d\mu_p(x)+ 
 \int_{A_{L,R}}|g(x\circ h)|^q d\mu_p(x).\end{equation}
 Let us now define $h:=h(x)=\gamma_x(t)$. Then by H\"older's inequality and the fact that $d$ is the control distance, i.e., we have $d(x,x\circ h)=d(h)\leq t$ (since $d(h):=d(e,h)$, by the definition of $d$, is the shortest ``time'' needed to go from $e$ to $h$ via a horizontal path), we obtain
\begin{align}
\int_{A_{L,R}}|g(x)-g(x\circ h)|^q d\mu_p(x)&=
\int_{A_{L,R}}\left| \int_0^t \frac{d}{ds}g(x\circ \gamma_x(s))ds\right|^q d\mu_p(x)\nonumber\\
&\leq \int_{A_{L,R}} t^{q} \int_0^t |\nabla_{\mathbb{G}}g(x\circ \gamma_x(s))|^q ds\, d\mu_p(x)\nonumber\\
&\leq\int_{A_{L,R}}  d(h)^{q} \int_0^t |\nabla_{\mathbb{G}}g(x\circ \gamma_x(s))|^q ds\, d\mu_p(x) \nonumber\\
&\lesssim R^q \int_{A_{L,R}}  \int_0^t |\nabla_{\mathbb{G}}g(x\circ \gamma_x(s))|^q ds\, d\mu_p(x) \nonumber\\
& \stackrel{\eqref{ineqN}}\lesssim R^q \int_{A_{L,R}}  \int_0^{cR}|\nabla_{\mathbb{G}}g(x\circ \gamma_x(s))|^q ds\, d\mu_p(x\circ \gamma_x(s)) \nonumber\\
&\lesssim R^{q+1} \int_{\mathbb{G}} |\nabla_{\mathbb{G}}g(x)|^q d\mu_p(x)\,\nonumber\\ 
&=C R^{q+1}\mu_p(|\nabla_{\mathbb{G}}g(x)|^q).
\label{eqn3005.1}
\end{align}
To estimate the second term in \eqref{NewEst1} we use \eqref{ineqN} and the estimate $(x\circ h)_{j_0} =(y(x)\circ (0,\ldots, c R^{\frac{1}{2\gamma}},\ldots, 0))\gtrsim R^{\frac{1}{2\gamma}}$ (which implies $N(x\circ h)\gtrsim R$),
so that an application of  Lemma  \ref{lemma1.2}  yields
\begin{align}
  \int_{A_{L,R}}|g(x\circ h)|^q d\mu_p(x)& \lesssim \frac{1}{R^{p}} \int_{A_{L,R}}|g(x\circ h)|^q |(x\circ h)_{j_0}|^{2\gamma} N(x\circ h)^{p-2\gamma} d\mu_p(x\circ h)\\
  & \lesssim \frac{1}{R^{p}} \int_{\mathbb{G}}|g(x)|^q |(x)_{j_0}|^{2\gamma} N(x)^{p-2\gamma} d\mu_p(x)\\
  & \leq \frac{A}{R^{p}} \mu_p(|\nabla_\mathbb{G} f|^q)+\frac{B}{R^p} \mu_p(|f-m|^q), \label{eqn3005.2}
\end{align}
which gives
\begin{equation}
    III\leq \left(\frac{A}{R^{p}}+CR^{q+1}\right) \mu_p(|\nabla_\mathbb{G} f|^q)+ \frac{B}{R^{p}}\mu_p(|f-m|^q)
\end{equation}
Finally, putting together the estimates for the three terms $I,II$ and $III$, we get
\begin{align*}
    &\left(1-\frac{B}{R^p}-\frac B R \right)\mu_p(|f-m|^q)\\
    &\leq \left(\frac A R+P_0(L_1)e^{aL_2^p}+\frac{A}{R^p}+C R^{q+1} \right)\mu_p(|\nabla_gf|^q),
\end{align*}
hence, by choosing $R$ and $L$ sufficiently large, with $L>R$, we get the result.
This concludes the proof.

\endproof

\begin{comment}
\begin{corollary}\label{cor.spectralgap}
    The positive self-adjoint operator 
    \[
\mathcal{L}_p:=-\Delta_{\mathbb{G}}+apN^{p-1}\nabla_{\mathbb{G}}N\cdot\nabla_{\mathbb{G}}\,,
    \]
    on $L^2(\mu_p)$ has a spectral gap.
\end{corollary}
\end{comment}

\begin{proof}[Proof of Corollary \ref{cor.spectralgap}]
    The proof is an immediate consequence of Theorem \ref{thm.SuffCond2} combined with the fact that, whenever \eqref{qPoinc2} holds true for some $q>1$, then it holds true also for $q'>q$ under the same unchanged probability measure; see Proposition 2.1.11 in \cite{Ing10}.
    
    %see Proposition 2.3 in \cite{BZ05}.
\end{proof}
\begin{remark}
    We want to point out that the choice of the homogeneous quasi-norm $N$ on the Carnot group provides a radical difference on the spectrum of the operators of the form \eqref{sch.par.case}. To be more precise, for a probability measure on an $H$-type group of the form $Z^{-1}e^{-aN^p}$, $p \in (1,2)$, the corresponding operator  \eqref{sch.par.case}  has empty essential spectrum when $N$ is the Kaplan norm, and does not even have a spectral gap when $N$ is the Carnot-Carath\'{e}odory distance; see Remark 4.5.4 in \cite{Ing10}.
\end{remark}

As a corollary of  Theorem \ref{thm.SuffCond2} one has that the global Poincar\'{e} inequalities of Theorem \ref{thm.SuffCond2} are satisfied by a family of measures whose potential is a perturbation of the one appearing in the density of $\mu_p$.

\begin{corollary}\label{Cor.petr}
Let $d\mu_W:=\tilde{Z}^{-1} e^{-W}d\mu_p$ be a probability measure with a potential $W$ which is Lipschitz (with respect to the Carnot-Carath\'eodory distance) satisfying
\begin{equation}\label{cor.poten}|\nabla_\mathbb{G}W(x)|^q\leq \delta N(x)^{p-2\gamma}|x_{j_0}|^{2\gamma}+\gamma_\delta,\quad \mbox{for almost all }  x\in\mathbb{G}, \end{equation}
for some $0<\delta \ll  1$ and $\gamma_\delta\in (0,\infty)$. Then the measure $\mu_W$ satisfies inequality \eqref{est.lemma1.2} for $p\geq 2\gamma$ and $q$ such that $\frac{1}{q}+\frac{1}{p}=1$.  Moreover, if there exists $\tilde{C}>0$ such that $W \leq \tilde{C}N$ then $\mu_{W}$ satisfies the global Poincar{\'e} inequality.
\begin{comment}
Moreover, if there exists $\tilde{C}$ such that $W\leq \tilde{C} N^p$ outside a ball centered at the origin, then $\mu_W$ satisfies the  global Poincar\'e inequality; i.e. we have
\[
\mu_{W}(|f-f_{\mu_p}|^q)\leq c_0 \mu_{W}(|\nabla f|^q)\,,
\]
for some constant $c_0$ independent of $f$.
\end{comment}
\end{corollary}
\proof
Let us start by replacing $f$ by $fe^{-\frac{W}{q}}$ in the inequality \eqref{est.lemma1.2}. This gives 
\begin{equation}
\label{gen,stat,1}
   \mu_{p}\left( e^{-W}|f|^q N^{p-2\gamma}|x_{j_0}|^{2\gamma} \right) \leq C \mu_{p} \left(|\nabla_{\mathbb{{G}}}(e^{-\frac{W}{q}}f)|^q\right)+D \mu_{p} \left( |e^{-\frac{W}{q}}f|^q \right) \,.
\end{equation}
Now, since
\begin{eqnarray}
\label{gen,stat,2}
|\nabla_{\mathbb{G}}(e^{-\frac{W}{q}}f)|^q & = & \left|\left(\nabla_{\mathbb{G}}e^{-\frac{W}{q}}\right)f+e^{-\frac{W}{q}}\nabla_{\mathbb{G}}f\right|^q\nonumber\\
& \leq & \left( \frac{|\nabla_{\mathbb{G}}W|}{q}|e^{-\frac{W}{q}}f|+e^{-\frac{W}{q}}|\nabla_{\mathbb{G}}f|\right)^q\nonumber\\
& \leq & C(q) \left(|\nabla_{\mathbb{G}}W|^q e^{-W}|f|^q+e^{-W}|\nabla_{\mathbb{G}}f|^q\right)\,,
\end{eqnarray}
substituting the latter in \eqref{gen,stat,1}, and using \eqref{cor.poten} we get
\begin{eqnarray*}
\mu_{W} \left(|f|^q N^{p-2\gamma}|x_{j_0}|^{2\gamma} \right) & \leq & C'(q) \mu_{W} \left( |\nabla_{\mathbb{G}}W|^q|f|^q\right) +C'(q) \mu_{W} ( |\nabla_{\mathbb{G}}f|^q)+D\mu_{W} (|f|^q)\nonumber\\
& \leq & \delta C'(q) \mu_{W} (N^{p-2\gamma}|x_{j_0}|^{2\gamma}|f|^q)+\gamma_{\delta} C'(q)\mu_{W} (|f|^q)\nonumber\\
&+& C'(q) \mu_{W}( |\nabla_{\mathbb{G}}f|^q)+D\mu_{W} (|f|^q)\,.
\end{eqnarray*}
Therefore, if $\delta$ is such that $1-\delta C'(q)>0$, the last inequality gives the result.\\

\indent Now to prove the global Poincar{\'e} for the measure $\mu_{W}$, we have to assume that $W$ is such that $W \lesssim N$. For $L>1$ and $R>0$ as in the proof of Theorem \ref{thm.SuffCond2}, we decompose the quantity $\mu_{W}(|f-m|^q)$ as follows
\begin{align}
\mu_{W}(|f-m|^q)=&\mu_{W}(|f-m|^q \boldsymbol{\mathds{1}} _{\{|x_{j_0}|^{2\gamma}N^{p-2\gamma}\geq R\}})+ \mu_W(|f-m|^q \boldsymbol{\mathds{1}} _{\{|x_{j_0}|^{2\gamma}N^{p-2\gamma}\leq R\}}\boldsymbol{\mathds{1}}_{\{N\leq L\}})\nonumber\\
+& \mu_W(|f-m|^q \boldsymbol{\mathds{1}} _{\{|x_{j_0}|^{2\gamma}N^{p-2 \gamma}\leq R\}}\boldsymbol{\mathds{1}}_{\{N\geq L\}})\,.\label{threecases,gen,}
\end{align}
 For the first and the third term of \eqref{threecases,gen,} one can proceed as in Theorem \ref{thm.SuffCond2}, while for the second term, arguing as in Theorem \ref{thm.SuffCond2}, we arrive at
\begin{eqnarray*}
\mu_{W}(|f-m|^q \boldsymbol{\mathds{1}}_{\{|x_{j_0}|^{2\gamma}N^{p-2\gamma} \leq R\}}\boldsymbol{\mathds{1}}_{\{N \leq L\}})& \leq & \frac{P_0(L_1)}{\tilde{Z}}\int_{\{N \leq L_2\}}|\nabla_{\mathbb{G}}f|^q\,dx\\
& \leq & \frac{P_0(L_1)}{\tilde{Z}}e^{aL_{2}^{p}+\tilde{C}L_2}\mu_{W} (|\nabla_{\mathbb{G}}f|^q)\,,
\end{eqnarray*}
 where $L_2>0$ is the one appearing in \eqref{equiv.CC}. In the last inequality we have used the fact that in $\{N \leq L_2 \}$ we have $W \leq \tilde{C}N \leq \tilde{C}L_2$. The proof is now complete.
\endproof

\section{Examples and general results}\label{section.Examples}
In this section we (I) describe classes of Carnot groups for which the sufficient condition of Theorem \ref{thm.SuffCond2}, expressed by  \eqref{grad-condition2}, is satisfied, and (II) construct a quasi-norm on a Carnot group $\mathbb G$ (with a formula depending on $\mathbb G$) which gives rise to the corresponding global Poincar\'e inequalities.
 This guarantees the existence of the spectral gap for the corresponding self-adjoint operator $\mathcal{L}_p$ as in Corollary \ref{cor.spectralgap}.

\subsection*{Homogeneous quasi-norms smooth away from the origin.} Before starting with the investigation of the validity of  \eqref{grad-condition2}, we briefly discuss here some homogeneous quasi-norms we will be using in the current section. The works \cite{Eg13}, \cite{Pop16}, \cite{LN06} and \cite{BFS18}, contain a recent exposition of examples of homogeneous quasi-norms on Carnot groups. 

Let us remark that the quasi-norm we have used so far is smooth away from the origin. For convenience we shall simply call such quasi-norms {\it smooth norms}, where the smoothness property shall be regarded as smoothness  everywhere except for the origin.

A natural example of such quasi-norm on a Carnot group on $\mathbb{R}^n$ equipped with the dilation $\delta_\lambda (x)=(\lambda^{\sigma_1} x_1,\ldots, \lambda^{\sigma_n} x_n)$ is given by the formula:
\begin{equation}\label{smoothN}
    N(x)=\left( \sum_{j=1}^na_j|x_j|^{2\beta_j}\right)^{1/\gamma},
\end{equation}
where $a_j$, for all $j=1,\ldots,n$, is a positive real number, and $\beta_j$, for all $j=1,\ldots,n$, is such that $\beta_j\in\mathbb{N}$ and $2\sigma_j\beta_j=\gamma\geq 2\sigma_n$. Homogeneous quasi-norms of the form \eqref{smoothN} are indeed natural extensions of the Euclidean norm on $\mathbb{R}^n$ to the dilated structure of the Carnot group. A particular case of a quasi-norm of the above form that one often encounters in the literature is the one with $\beta_j=\sigma_n!/\sigma_j$ and $\gamma=2\sigma_n!$

Given a Carnot group on $\mathbb{R}^n$ as above, one can use quasi-norms of the form \eqref{smoothN} on groups of variables to give rise to other homogeneous quasi-norms. Explicitly, we can define the sum
\begin{equation}\label{generalsmoothN}
\tilde{N}(x) :=  \left(\sum_{j=1}^m N_j(x_{j_1},x_{j_2},\ldots,x_{j_{k_j}})^{\alpha_j}\right)^{1/\alpha},
\end{equation}
where ${k_j}\leq n$, $j_1,\ldots, j_{k_j}\in\{1,\ldots,n\}$, $m\in\mathbb{N}$, and  $N_j$ are quasi-norms of the form \eqref{smoothN} defined on a subspace of $\mathbb{R}^n$ \footnote{To be precise, the quasi-norms $N_j$ above are
$$N_j(x_{j_1},x_{j_2},\ldots,x_{j_{k_j}})=\left( \sum_{i=1}^{k_j}a_{ji}|x_{j_i}|^{2\beta_i}\right)^{1/\gamma_j},$$
where the parameters are as in \eqref{smoothN}, that is such that $\beta_i/\sigma_{j_i}=\gamma_j\geq 2\sigma_{j_{k_j}}$ for every $i=1,\ldots, k_j$, while $a_{ji}$ are positive real numbers.}. In \eqref{generalsmoothN} we have chosen $\alpha\in \mathbb{N}$, and thus the homogeneity requires that the $\alpha_j$'s are such that $\alpha_j/\gamma_j=\alpha$, where $\gamma_j$ is the exponent in $N_j$ as in \eqref{smoothN}. 

In Carnot groups $\mathbb{G}$ of step two, the general formula \eqref{generalsmoothN} boils down to the quasi-norm given in \eqref{smootNstep2} below, which can be viewed as a generalization of the \textit{Kaplan norm} on H-type groups (see \cite{Ing10}). Precisely, if $\mathbb{G}$ is a group on $\mathbb{R}^{m+n}$, with $m$ being the number of generators, then denoting by $(x,t)\in\mathbb{R}^m\times\mathbb{R}^n$ an element of the group, one has that a class of smooth quasi-norms on $\mathbb{G}$ is given by
\begin{equation}\label{smootNstep2}
    N_\alpha(x,t)=\left(\|x\|^{4\alpha}+\sum_{j=1}^nc_j t_j^{2\alpha} \right)^{\frac{1}{4\alpha}},
\end{equation}
where $\alpha$ is a positive integer,  $c_j$, for every $j=1,\ldots,n$, is a positive real number,   and 
 $\|x\|:=\left(\sum_{k=1}^m x_k^2\right)^{1/2}$.

Let us also remark that all the quasi-norms listed above are as in the hypothesis of Theorem \ref{thm.SuffCond2}.

Considering the norms defined above, we will now focus on examples of groups where these norms allow to recover global Poincaré inequalities. We stress once more that spectral gaps for suitable corresponding operators follow from the 2-Poincaré inequality. 

In what follows we will first show that on step two Carnot groups, by choosing $N$ as in \eqref{smootNstep2}, we have that our sufficient condition \eqref{grad-condition2} holds true, and, consequently, so does also the corresponding global Poincaré inequality for the suitable probability measure.

Next, for a wide class of Carnot groups of arbitrary step,  we will show that condition \eqref{grad-condition2} is satisfied if $N$ is as in \eqref{smoothN} or, more generally, as in \eqref{generalsmoothN}. This, once again, will imply the validity of global Poincar\'{e} inequalities, and, as before, of the spectral gaps for operators as in \eqref{sch.par.case}, or, more precisely, as in Corollary \ref{cor.spectralgap}.
%Formula \eqref{generalsmoothN} can be viewed as  the most general description of smooth norms one can give on general Carnot groups; see Proposition \ref{prop.smooth.hom}.
\\

\subsection*{Carnot groups of step 2}
We start with the investigation  of step 2 Carnot groups.
Below we shall use the notation $(\mathbb{R}^{m+n},\circ)$ for an $N=m+n$-dimensional Carnot group of step 2 with $m$ generators and composition law $\circ$. A point in $\mathbb{G}=(\mathbb{R}^{m+n},\circ)$ will be denoted by $(x,t)$, with $x\in\mathbb{R}^m$ and $t\in\mathbb{R}^n$.

We recall that any $N$-dimensional Carnot group of step 2 and $m$ generators is naturally isomorphic to a step two Carnot group $(\mathbb{R}^{m+n},\circ')$, where $n=N-m$, the composition law is
\begin{equation}\label{compLaw2}
    (x,t)\circ' (\xi,\tau)=(x+\xi, t_1+\tau_1+\frac 1 2 \langle B^{(1)}x,\xi\rangle, \ldots, t_n+\tau_n+\frac 1 2 \langle B^{(n)}x,\xi\rangle ),
\end{equation}
and  $B^{(j)}$, for all $j=1,\ldots, n$, is an $m\times m$ skew-symmetric matrix. 
Therefore, without loss of generality, hereafter we consider two step Carnot groups with a composition law of the form $\circ'$ defined through some skew-symmetric matrices $B^{(j)}$'s.

Notice that, for a Carnot group of step 2 with $m$ generators on $\mathbb{R}^{m+n}$, \eqref{compLaw2} gives that
\begin{equation}\label{VFstep2}
    X_j=\partial_{x_j}+\frac1 2 \sum_{k=1}^n\sum_{i=1}^m B_{ij}^{(k)}x_i \, \partial_{t_k}, \quad \text{for all}\quad  j=1,\ldots, m.
\end{equation}
Formula \eqref{VFstep2} will be very useful to prove the following proposition.
\begin{proposition}\label{prop.2stepG}
Let $\mathbb{G}=(\mathbb{R}^{n_1+n},\circ')$ be a Carnot group of step 2 and $n_1$ generators, and let $N_\alpha$ be a smooth quasi-norm on $\mathbb{G}$ as in \eqref{smootNstep2}, where $\alpha$ can be any positive integer. Then, for every $x\in\mathbb{G}\setminus \{0\}$ there exists a constant $C>0$ such that
\begin{equation}
    |\nabla N_\alpha (x)|\geq C\frac{\|(x_1,\ldots,x_{n_1})\|^{4\alpha-1}}{N_\alpha(x)^{4\alpha-1}},
\end{equation}
where $\|\cdot\|$ stands for the Euclidean norm on $\mathbb{R}^{n_1}$.
\end{proposition}
\proof
Recall that, for a Carnot group $\mathbb{G}$ of step $2$, the generators $X_1,\ldots, X_{n_1},$ are of the form \eqref{VFstep2}. Therefore, since
\begin{align*}
    |X_j N_\alpha (x)|^2=& \left( \frac{8 \|x\|^{4\alpha-2}x_j+ \sum_{k=1}^n\sum_{i=1}^{n_1} B_{ij}^{(k)}x_i c_k t_k }{4 N_\alpha(x)^{4\alpha-1}}\right)^2\\
    =& \frac{64 \|x\|^{8\alpha-4}x_j^2+ (\sum_{k=1}^n\sum_{i=1}^{n_1}  B_{ij}^{(k)}x_i c_kt_k)^2 +2\sum_{k=1}^n\sum_{i=1}^{n_1} B_{ij}^{(k)}x_i x_jc_kt_k}{ 16 N_\alpha(x)^{8\alpha-2}},
\end{align*}
and since $\sum_{i,j=1}^{n_1}  B_{ij}^{(k)}x_i x_j=0$ by the skew-symmetry of $B^{(k)}$, we have
\begin{align*}
    |\nabla N_\alpha(x)|^2\geq&\frac{4\|x\|^{8\alpha-2}}{N_\alpha(x)^{8\alpha-2}}, 
\end{align*}
 and the proof is complete.
\endproof
\begin{remark}
Thanks to Proposition \ref{prop.2stepG} and Theorem \ref{thm.SuffCond2} we have that global $q$-Poincaré inequalities hold true on step 2 Carnot groups for probability measures whose density is of the form $e^{-a N_\alpha^p}$, with $N_\alpha$ as in \eqref{smootNstep2} and $q$ conjugate exponent of $p$. Moreover, the $2$-Poincaré inequality yields the validity of the spectral gap for the operator $\mathcal{L}_p:=-\Delta_{\mathbb{G}}+p N_\alpha^{p-1} \nabla_\mathbb{G}N\cdot \nabla_\mathbb{G}$, for $p$ as in $\mu_p$.
\end{remark}
Summarizing, we have proved the following theorem.

\begin{theorem}\label{thm.2step}
Let $\mathbb{G}$ be a Carnot group of step 2, $\alpha$ a positive integer, and $N_\alpha$ a homogeneous quasi-norm on $\mathbb{G}$ as in \eqref{smootNstep2}. Then 
$$\mu_p(|f-f_{\mu_p}|^q)\leq \mu_p(|\nabla f|^q)$$
for every $p\geq 8\alpha$, and with $q$ being the conjugate exponent of $p$.
\end{theorem}

\begin{example}
    [$\mathbb{H}$-type groups.]
Note that, for a Carnot group of step 2 of $\mathbb{H}$-type on $\mathbb{R}^{m+n}$, the quasi-norm $N_\alpha$ in \eqref{smootNstep2} with $\alpha=1$ and $c_j=\frac{1}{16}$, for all $j=1,\ldots,n$, coincides with the so called Kaplan norm
$$N_1(x)=(\|x\|^4+\frac{1}{16}|z|^2)^{1/4}, \quad x=(w,z)\in\mathbb{R}^n\times\mathbb{R}^m.$$
In \cite{Ing10} the author showed that, on $\mathbb{H}$-type groups, $|\nabla  N_1(x)|=\frac{\|x\|}{N(x)}$, and that global $q$-Poincaré inequalities with respect to $d\mu_p=\frac{1}{Z}e^{-a\, N_\alpha^p(x)}dx$,  with $a>0$, $p\geq 2$, and $q$ conjugate exponent of $p$, hold true. We remark that
the previous identity also gives 
$$|\nabla N_1(x)|=\frac{\|x\|}{N(x)}=\frac{\|x\|^3}{\|x\|^2 N(x)}\geq \frac{\|x\|^3}{N(x)^3},$$
which is condition \eqref{grad-condition2} in Theorem \ref{thm.SuffCond2} giving the validity of global Poincaré inequalities. 

Moreover, here Theorem \ref{thm.2step} applies and generalizes, in some sense, the result in \cite{Ing10}, allowing to conclude global Poincaré inequalities for the class of probability measures defined as $d\mu_p=\frac{1}{Z}e^{-a\, N_\alpha^p(x)}dx$, for $a>0$ and $\alpha$ being any positive integer.
\end{example}

\begin{example}[The anisotropic Heisenberg group.]
We conclude this part dedicated to step 2 groups by
considering a group treated in \cite{BDZ22} to which our results apply. The group under consideration is the so called anisotropic Heisenberg group $\mathbb{H}_{2n}\left( \frac 1 2, 1\right)$. Since $\mathbb{H}_{2n}\left( \frac 1 2, 1\right)$ is a Carnot group of step 2, we get the validity of global $q$-Poincaré inequalities for any smooth homogeneous quasi-norm as in \eqref{smootNstep2} by Theorem \ref{thm.2step}. However, the validity of such inequalities was first proved in \cite{BDZ22} by using probability measures whose density depends on a specific smooth quasi-norm $N$, that is, for $N$ being the fundamental solution for the sub-Laplacian. The quasi-norm in \cite{BDZ22} is 
$$N(x)=\frac{(B^2+t^2)^{1/4n}(AB+t^2+A\sqrt{A^2+B^2})^{1/2-1/4n}}{(B+\sqrt{B^2+t^2})^{1/2}}, $$
where
$$A=\frac{x_1^2}{2}+\frac{x_{n+1}^2}{2}+\frac 1 2\sum_{\substack{j=1\\j\neq n+1}}^{2n}x_j^2, \quad \text{and}\quad
B=\frac{x_1^2}{4}+\frac{x_{n+1}^2}{4}+\frac 1 2\sum_{\substack{j=1\\j\neq n+1}}^{2n}x_j^2.$$
Notice that the quasi-norm $N$ is smooth away from $0$. Moreover, in \cite{BDZ22} the authors proved that, for all $x\neq 0_{\mathbb{G}}$,
$$|\nabla N(x)|\geq C \frac{\|x\|^2}{N^2}\geq C\frac{\|x\|^{\gamma-1}}{N^{\gamma-1}}, \quad \forall \gamma\geq 3, $$
where $\|x\|$ is the Euclidean norm of $x=(x_1,\ldots, x_{2n})$, therefore the sufficient condition \eqref{grad-condition2} is satisfied and Theorem \ref{thm.SuffCond2} applies for all $\gamma\geq 3$. 

Summarizing, on the anisotropic Heisenberg group one can apply both Theorem \ref{thm.SuffCond2} (with $N$ as in \cite{BDZ22}) and Theorem \ref{thm.2step}, and get, in the first case, the same result as in \cite{BDZ22}, while, in the second case, Poincar\'{e} inequalities with respect to different probability measures defined through $N_\alpha$  as in \eqref{smootNstep2}. 

Let us finally remark that in \cite{BDZ22} the authors had to deal with the non trivial problem of finding a fundamental solution of the  sub-Laplacian in order to find the suitable probability measure to prove the inequalities. Theorem \ref{thm.2step}, instead, is direct, and gives already a class of measures for which the inequalities are true. 
 \end{example}

\subsection*{Carnot groups of step $r\geq 2$}
Besides Carnot groups of step 2, there are other groups to which our result applies, that is groups such that the sufficient condition for the Poincar\'e inequality is verified by any homogeneous quasi-norm on the group being smooth away from the origin and of the form \eqref{smoothN} or \eqref{generalsmoothN}. Such groups include those described in the following result.

\begin{lemma}\label{unit.vf}
Let $\mathbb{G}$ be a Carnot group of step $r$ on $\mathbb{R}^n$, and let $\{X_j\}_{1\leq j\leq n_1}$ be the generators of the first stratum $V_1$. If there exists $j_0\in \{1,...,n_1\}$ such that
\begin{equation}\label{existence.j0}
    X_{j_0}=\partial_{x_{j_0}},
\end{equation}
then \eqref{grad-condition2} holds with $j=j_0$ and with $N$ as in \eqref{smoothN} or \eqref{generalsmoothN}. 
\end{lemma}

\proof
We give the proof for $N$ as in \eqref{smoothN}, since for quasi-norms of the form \eqref{generalsmoothN} one can proceed similarly.

Note that, due to the form  of $X_{j_0}$ and of the quasi-norm $N$ in \eqref{smoothN},
\begin{align*}
    |X_{j_0}N(x)|=&\bigg|\frac{2\beta_{j_0}a_{j_0}x_{j_0}^{2\beta_{j_0}-1}}{\gamma N^{\gamma-1}}\bigg|=C \frac{|x_{j_0}|^{\gamma-1}}{N^{\gamma-1}}.
\end{align*}
 Therefore, since $|\nabla_\mathbb{G}N(x)|\geq |X_{j_0}N(x)|$ for all $x\in\mathbb{G}$, the previous estimate amounts to \eqref{grad-condition2} with $j=j_0$. 

For combinations of smooth quasi-norms as in \eqref{generalsmoothN}, by repeating the same considerations as above we trivially get the same result. This concludes the proof.
\endproof

By using Lemma \ref{unit.vf}, one obtains the following %trivial 
corollary of Theorem \ref{thm.SuffCond2}.

\begin{proposition}
    \label{prop.unit.vf}
Let $\mathbb{G}$ be a Carnot group of step $r$ on $\mathbb{R}^n$ and let $\{X_j\}_{1\leq j\leq n_1}$ be the generators of the first stratum $V_1$. If there exists $j_0\in \{1,...,n_1\}$ such that \eqref{existence.j0} holds,
then the hypotheses of Theorem \ref{thm.SuffCond2} are satisfied for $N$ as in \eqref{smoothN} or \eqref{generalsmoothN}.
\end{proposition}

Among the groups to which Lemma \ref{unit.vf} applies, and therefore satisfying conditions \eqref{existence.j0} and \eqref{grad-condition2}, for $N$ as in \eqref{smoothN} or \eqref{generalsmoothN}, we have:
\begin{itemize}
    \item Carnot groups of ''Engel type'', see e.g.  \cite{C20}, \cite{C22}, on $\mathbb{R}^n$ (i.e. Carnot groups of filiform-type), for $n\geq 4$, with polynomial coordinates \footnote{Here we use polynomial coordinates arising from a strong Malcev basis of the Lie algebra of the group; see \cite{CG} for details about such coordinates.} (see \cite{CFZ24});  
    \item The Cartan group (see \cite{Dix57});
    \item {\it Kolmogorov}-type groups (see \cite{BLU07});
    \item Carnot groups arising from some Sub-Laplacians, like, for instance, the ones arising from the lifting of Bony-type Sub-Laplacians and those related to some Sub-Laplacian arising in control theory (see Section 4.3 in \cite{BLU07} for details);
    \item Sums of Carnot groups of the previous type.\\
\end{itemize}

A combination of Lemma \ref{unit.vf}, Proposition \ref{prop.unit.vf} and Theorem \ref{thm.SuffCond2}, leads to the proof of Theorem \ref{THM:hypothesis}. For completeness, we give the proof in the following.
\proof[Proof of Theorem \ref{THM:hypothesis}] By the result of Helffer and Nourrigat in \cite[page 99]{HN85}, see also \cite[Theorem 2.2]{GG90}, there exists an admissible change of coordinates such that one of the generating vector fields, say $X_1$ after suitably relabeling the generating vector fields, is sent into the vector field $Y_1=\partial_{x_1}$, while $X_2,\ldots, X_{n_1}$ are sent into suitable vector fields $Y_2,\ldots,Y_{n_1}$. Since $\{Y_j\}_{j=1,\ldots,n_1}$ are generating vector fields of a stratified Lie algebra (isomorphic to the starting one) satisfying \eqref{existence.j0} with $j_0=1$, by Proposition \ref{prop.unit.vf} we conclude that, possibly after a change of coordinates, Theorem \ref{thm.SuffCond2} holds with $N$ as in \eqref{smoothN}, which concludes the proof.
\endproof

\begin{remark}
    We note that Theorem \ref{THM:hypothesis} holds true also for $\mu_p$ corresponding to $N$ given by \eqref{generalsmoothN}. In this case we have $p \geq 8\alpha$.
\end{remark}

\endproof
\nocite{*}

\section*{Declarations}

\subsection*{Author contributions}
All authors have equally contributed.

\subsection*{Funding}
 Marianna Chatzakou
is supported by the FWO Odysseus 1 grant G.0H94.18N: Analysis and Partial Differential Equations,
and by the Methusalem programme of the Ghent University Special Research Fund (BOF) (Grant
number 01M01021), and is a postdoctoral fellow of the Research Foundation – Flanders (FWO) under
the postdoctoral grant No 12B1223N. Serena Federico has received funding from the European Union's
Horizon 2020 research and innovation programme under the Marie Sklodowska-Curie grant agreement
No 838661.

\subsection*{Data availibility}
Data sharing is not applicable to this article as no datasets were generated or analysed
during the current study.

\subsection*{Conflict of interest} The authors declare that they have no competing interests.

\subsection*{Ethical approval} Not applicable

\end{document}